\pdfoutput=1
\documentclass[12pt]{amsart}
\usepackage{epsfig,amssymb,pdfsync,color}

\newtheorem{Thm}[equation]{Theorem}
\newtheorem{Cor}[equation]{Corollary}
\newtheorem{Lem}[equation]{Lemma}
\newtheorem{Pro}[equation]{Proposition}

\theoremstyle{definition}

\newtheorem{Def}[equation]{Definition}

\theoremstyle{remark}

\newtheorem{Rem}[equation]{Remark}

\numberwithin{equation}{section}
\makeatletter
\renewcommand{\c@figure}{\c@equation}
\makeatother

\newcommand{\maths}[1]{{\bf #1}}

\newcommand{\T}{{\mathcal T}}

\newcommand{\RR}{\maths{R}}
\newcommand{\NN}{\maths{N}}

\renewcommand{\SS}{\maths{S}}

\newcommand{\ra}{\rightarrow}

\newcommand{\bord}{\partial}

\newcounter{fig}

\newcommand{\ack}{\noindent{\bf Acknowledgement.}}

\begin{document}

\title[Dirichlet Problem]{Addendum: the case of closed surfaces.\\  (Boundary Value Problems on Planar Graphs and Flat Surfaces with integer cone singularities, I: The Dirichlet Problem)}

\author{Sa'ar Hersonsky}

\address{Department of Mathematics\\ 
University of Georgia\\ 
Athens, GA 30602}

\urladdr{http://saarhersonsky.wix.com/saar-hersonsky}
\email{saarh@math.uga.edu}
\thanks{This work was partially supported by a grant from the Simons Foundation (\#319163 to Saar Hersonsky).}
\date{\today}

\begin{abstract}

We extend our discrete uniformization theorems  for planar, $m$-connected, Jordan domains  [Journal f\"ur die reine und angewandte Mathematik 670 (2012), 65--92] to closed surfaces of genus $m\geq 1$.

\end{abstract}

\maketitle

\section{Introduction}
\label{se:Intro}

The aim of this addendum to our paper \cite{Her1} is to extend 
the results therein to the case of  closed, triangulated  surfaces, of any non-positive genus. The new ingredient we employ is the {\it Hauptvermutung} for $2$-manifolds with boundary.  
The extension provided here is natural and  straightforward. In fact, Corollary 0.3 in the above mentioned paper already provided  a discrete unifomization theorem for closed surfaces; those that are  obtained by doubling an $m$-connected ($m>1$), planar, Jordan domain along its full boundary. 



\smallskip

By a singular flat surface, we will mean a surface which carries a metric structure locally modeled on the Euclidean plane, except at a finite number of points. These points have cone singularities, and the cone angle is (in general) allowed to take any positive value (see for instance \cite{Tro} for a detailed survey).  Following the convention  in \cite{Her1},  a Euclidean rectangle will denote the image under an isometry of a planar Euclidean rectangle, and  a  singular flat, genus zero compact surface with $m\geq 3$ boundary components, will be called a {\it ladder of singular pairs of pants}.

\smallskip
In order to make the statement of the theorem from \cite{Her1} that will be used here self-contained, recall that in \cite{Her1} we considered a planar, bounded, $m$-connected region $\Omega$, where  $\bord\Omega$ denotes its piecewise linear boundary.  Let $\mathcal{T}$ be a triangulation of $\Omega\cup\bord\Omega$. Let $\bord\Omega=E_1\sqcup E_2$, where $E_1$ is the outermost component of $\bord\Omega$. We invoke a {\it conductance function} on ${\mathcal T^{(1)}}$, making it a {\it finite network}, and use it to define a {\it combintorial Laplacian} $\Delta$ on ${\mathcal T}^{(0)}$. 

\smallskip

For any positive constant $k$, let $g$ be the solution of a {\it discrete Dirichlet boundary value problem}  defined on ${\mathcal T}^{(0)}$ which is determined by  requiring that $g|_{{\mathcal T}^{(0)}\cap E_1}=k, g|_{{\mathcal T}^{(0)}\cap E_2}=0$, and that
$\Delta g=0$ at every interior vertex of ${\mathcal T}^{(0)}$. 
Furthermore, let $E(g)$ be the {\it Dirichlet energy} of $g$, and let $\frac{\bord g}{\bord n}(x)$ denote the {\it normal derivative} of $g$ at the vertex $x\in\bord \Omega$ (cf. \cite[Section 1]{Her1} for further details).

\smallskip


\smallskip

We now recall  one of the main results of \cite{Her1}.

\begin{Thm} {\mbox {\rm \cite[Theorem 0.2]{Her1}}}
  \label{Th:ladder}
 Let  $(\Omega,\bord\Omega,{\mathcal T})$ be a bounded, $m$-connected, planar, Jordan region with $E_2=E_2^1\sqcup E_2^2\ldots \sqcup E_2^{m-1}$.  Let $S_{\Omega}$ be a ladder of singular pairs of pants such that
 \begin{enumerate}
 \item ${\mbox{\rm Length}_{\rm Euclidean}}(S_{\Omega})_{E_1}= \ \ \ \sum_{x\in E_1}\frac{\bord g}{\bord n}(x)$, and
  \item ${\mbox{\rm Length}_{\rm Euclidean}}(S_{\Omega})_{E_2^i}=- \sum_{x\in E_2^i}\frac{\bord g}{\bord n}(x)$, for $i=1,\cdots, m-1$,
 \end{enumerate}
 where $(S_{\Omega})_{E_1}$ and $(S_{\Omega})_{E_2^i}$, for $i=1,\cdots, m-1$, are the boundary components of $S_{\Omega}$.
   Then, there exists
   a mapping $f$ which associates to each edge in  ${\mathcal  T}^{(1)}$ a unique  Euclidean rectangle in $S_{\Omega}$ in such a way that the collection of these rectangles forms a tiling of $S_{\Omega}$. 
   Furthermore, $f$ is  boundary preserving, and $f$ is energy preserving in the sense that  $E(g)= {\rm Area}(S_{\Omega})$. 
 \end{Thm}
In the statement of the theorem,  ``boundary preserving" means that the rectangle associated to an edge $[u,v]$ with $u\in \partial \Omega$ has one of its edges on a corresponding boundary component of the singular surface. 

\smallskip 
Given a domain  as in the theorem, one may look at the {\it closed surface} 
 obtained by doubling the domain along its full  boundary. The following Corollary provides discrete uniformization for this class of surfaces.
 
  \begin{Cor}{ \mbox{\rm \cite[Corollary 0.3]{Her1} } } 
 \label{Cor:surface}
Under the assumptions of Theorem~\ref{Th:ladder}, there exists a canonical pair $(S,f)$, where $S$ is a flat surface with conical singularities of genus $(m-1)$ tiled by Euclidean rectangles, and $f$ is an energy preserving mapping from
 ${\mathcal T}^{(1)}$ into $S$, in the sense that $2E(g)={\rm Area}(S)$.   Moreover, $S$ admits a pair of pants decomposition whose dividing curves have Euclidean lengths given by ${\rm (1)-(3)}$ of Theorem~\ref{Th:ladder}.
 
 \end{Cor} 
 \begin{proof}
 
 Given $(\Omega,\bord\Omega,{\mathcal T})$, glue together two copies of $S_{\Omega}$ (their existence is guaranteed by Theorem~\ref{Th:ladder}) along corresponding boundary components.  This results in a flat surface $S=S_{\Omega}\bigcup\limits_{ \bord\Omega }S_{\Omega}$ of genus $(m-1)$ and a mapping ${\bar f}$ which restricts to $f$ on each copy.
 
\end{proof}
 
\medskip
We are now ready to address the case of an arbitrary closed surface of genus $m\geq 1$,  the main purpose of this addendum. Henceforth, we will use  the following standard notation. If $K$ is a complex, then $|K|$ denotes the union of the elements in $K$, endowed with the subspace topology induced by the topology on $R^3$.
 
\smallskip

 \begin{Thm}[Discrete Uniformization of Closed Surfaces]
 \label{Th:closedsurface}
Let  $(X,{\mathcal T})$ be a closed, triangulated surface of genus $m$, $m\geq 1$.  Then there exists a mapping $f$ which associates to each edge in a refined triangulation of ${\mathcal  T}^{(1)}$ a unique  Euclidean rectangle in a singular flat surface which is homeomorphic to $X$ and denoted by $S$. Each one of the cone singularities in $S$ is an integer multiple of $\pi$. Moreover,  the collection of these rectangles forms a tiling of $S$. Finally, $S$ admits a pair of pants decomposition whose dividing curves have Euclidean lengths given by ${\rm (1)-(3)}$ of Theorem~\ref{Th:ladder}.
 
\end{Thm} 
\begin{proof}
We first assume that $X$, considered now as a simplicial complex,  is a surface with negative Euler characteristic $\chi(X)= 2 - 2m$.  
Let $\gamma$ be a disjoint union of $m+1$ embedded, closed, $1$-cycles in ${\mathcal T}^{(1)}$ such that if  $(S, {\mathcal T})$ is cut along $\gamma$, the resulting pieces 
$(P_1,{\mathcal T}_1)$ and  $(P_1,{\mathcal T}_1)$,  are triangulated, genus $0$, pair of pants, each of which having $m+1$ boundary components; where  ${\mathcal T}_1={\mathcal T}|P_1$ and ${\mathcal T}_2={\mathcal T}|P_2$. Note that the boundary components of $(P_1,{\mathcal T}_1)$ and  of $(P_1,{\mathcal T}_1)$ are in one to one correspondence with the components of $\gamma$.

\smallskip

The classification of $2$-manifolds with boundary (see for instance \cite[section 5]{Moi}) 
 implies that  $|P_1|$ and $|P_2|$ are homeomorphic. Hence, it follows from the Hauptvermutung in dimension $2$  
(see for instance \cite{Moi, Pap})
that $P_1$ and $P_2$ are combinatorially equivalent.  Henceforth, let $L_1$ be a subdivision of ${\mathcal T}_1$ and $L_2$ a subdivision of ${\mathcal T}_2$, such that  $(P_1, L_1)$ and  $(P_2, L_2)$ are combinatorially isomorphic.

\smallskip

Let us choose the same conductance constants,  ${\mathcal C}$, for $L_1$ and $L_2$. From a topological perspective, only the {\it planarity} of the domain  $\Omega$  was used in the proof of Theorem~\ref{Th:ladder}. Hence, we may apply Theorem~\ref{Th:ladder} to the (isomorphic) networks $(P_1, L_1)$ and $(P_2, L_2)$ respectively. Once the conductance constants are given,  the assertions of Theorem~\ref{Th:ladder} are determined solely by the combinatorial isomorphism class of the triangulation. Therefore, it follows that the two  ladder of singular pair of pants $S_{P_1}$  and $ S_{P_2}$, whose existence is guaranteed by Theorem 0.1, are in fact identical. 

\smallskip

Therefore, in this case the assertions of Theorem~\ref{Th:closedsurface} now follow by setting $S$ to be the singular flat surface obtained by gluing the two isometric singular pair of pants 
$S_{P_1}$  and $ S_{P_2}$ along matching boundary components. Also, \cite[Equation (4.12)]{Her1} justifies the assertion  regarding  the cone angles.

\smallskip

We now treat the case of $m=1$, i.e.,  $(X, {\mathcal T})$ is a triangulated torus.  Choose in ${\mathcal T}^{(1)}$
an embedded, piecewise-linear, $1$-cycle, $\tau$, that is a meridian. Hence, when $(X,{\mathcal T})$ is cut along $\tau$, the result is a triangulated cylinder denoted by $S_{\tau}$. Let $\tau_1$ and $\tau_2$ be the two boundary components of $S_{\tau}$.  Theorem 0.4 in \cite{Her1} provides a discrete uniformization in the case of a triangulated annulus. Since only the planarity of the annulus was used in the proof, we may apply it to the case of a triangulated cylinder. In particular,  the proof shows that the image of the cylinder $S_{\tau}$ under the mapping $f$ is a straight Euclidean cylinder whose height equals $k$ and whose circumference equals 
$$
\sum_{x\in\tau_1} \frac{\partial g}{\partial n}(x).
$$
Hence, in this case the assertions of Theorem~\ref{Th:closedsurface} follow by letting $S$ be the flat torus obtained by gluing the top and the bottom of this 
Euclidean cylinder thus obtained, by an isometry.
 
\end{proof}
 
 \ack\  We thank  Ed Chien and Feng Luo for expressing their interest in our work, and for a sketch of their research 
 plan to address (following ideas of Danny Calegari) a different path to tiling of closed surfaces by squares.
 

\begin{thebibliography}{99}




\bibitem{Her1} S.~Hersonsky, \emph{Boundary Value Problems on Planar Graphs  and Flat Surfaces with Integer Cone singularities I:  The Dirichlet problem}, Journal f\"ur die reine und angewandte Mathematik 670 (2012), 65--92.



\bibitem{Moi} E.D.~Moise, \emph{Geometric Topology in Dimensions $2$ and $3$},  Graduate Texts in Mathematics, Vol. 47. Springer-Verlag, New York-Heidelberg, 1977.

\bibitem{Pap} Ch.~Papakyriakopoulos, \emph{A new proof of the invariance of the homology
groups of a complex}, Bull. Soc. Math. Gr\`{e}ce 22, (1943), 1Ð154. 

\bibitem{Tro} M.~Troyanov, \emph{On the moduli space of singular Euclidean surfaces}, Handbook of Teichm\"{u}ller theory, Vol. I,  (2007), 507--540, Eur. Math. Soc., Z\"urich.
\end{thebibliography}
\end{document}